\documentclass[a4paper,11pt]{article}

\usepackage{amsmath}
\usepackage{amssymb}
\usepackage{amsthm}
\usepackage{graphicx}
\usepackage{amscd}
\usepackage{epic, eepic}
\usepackage{url}
\usepackage[dvipsnames]{xcolor}
\usepackage[utf8]{inputenc} 
\usepackage{comment}
\usepackage{enumerate}   
\usepackage{graphicx}
\usepackage{epstopdf}
\usepackage{enumitem}
\usepackage{todonotes}
\usepackage{soul}

\usepackage[top=22mm, bottom=22mm, left=20mm, right=20mm]{geometry}
\usepackage[bookmarks=false,hidelinks]{hyperref}
\usepackage[capitalize]{cleveref}
\usepackage[normalem]{ulem}
\newcommand{\stkout}[1]{\ifmmode\text{\sout{\ensuremath{#1}}}\else\sout{#1}\fi}
\usepackage{tikz}
\usepackage[allcommands]{overarrows}
\usetikzlibrary{math,calc,intersections}
\usetikzlibrary{positioning,arrows,shapes,decorations.markings,decorations.pathreplacing, decorations.pathmorphing,matrix,patterns}
\tikzstyle{vertex}=[circle,draw=black,fill=black,inner sep=0,minimum size=5pt,text=white,font=\footnotesize]

\makeatletter
\renewenvironment{proof}[1][\proofname] {\par\pushQED{\qed}\normalfont\topsep6\p@\@plus6\p@\relax\trivlist\item[\hskip\labelsep\bfseries#1\@addpunct{.}]\ignorespaces}{\popQED\endtrivlist\@endpefalse}
\makeatother

\newtheorem{theorem}{\bf Theorem}[section]
\newtheorem{lemma}[theorem]{\bf Lemma}
\newtheorem{claim}[theorem]{\bf Claim}

\newtheorem{question}[theorem]{\bf Question}
\newtheorem{problem}[theorem]{\bf Problem}

\theoremstyle{definition}

\newcommand\claimproofend{\renewcommand{\qedsymbol}{$\boxdot$}
\end{proof}
\renewcommand{\qedsymbol}{$\square$}}

\def\eps{\varepsilon}

\def\cA{\mathcal{A}}
\def\cC{\mathcal{C}}

\def\cF{\mathcal{F}}

\title{\vspace{-0.9cm}Universality for transversal Hamilton cycles in random graphs}
\author{Micha Christoph\thanks{Department of Mathematics, ETH Z\"{u}rich, Switzerland.
 E-mail: {\tt micha.christoph@math.ethz.ch}.
 Supported by the SNSF Ambizione Grant No. 216071.
}\and 
Anders Martinsson\thanks{Department of Computer Science, ETH Z\"{u}rich, Switzerland.
 E-mail: {\tt anders.martinsson@inf.ethz.ch}}\and
Aleksa Milojevi\'{c}\thanks{Department of Mathematics, ETH Z\"urich, Switzerland. E-mail: {\tt aleksa.milojevic@math.ethz.ch}. Supported by SNSF Grant 200021-228014.}
}
\date{}

\begin{document}

\maketitle
\begin{abstract}
    A tuple $(G_1,\dots,G_n)$ of graphs on the same vertex set of size $n$ is said to be Hamilton-universal if for every map $\chi: [n]\to[n]$ there exists a Hamilton cycle whose $i$-th edge comes from $G_{\chi(i)}$. Bowtell, Morris, Pehova and Staden proved an analog of Dirac's theorem in this setting, namely that if $\delta(G_i)\geq (1/2+o(1))n$ then $(G_1,\dots,G_n)$ is Hamilton-universal. Combining McDiarmid's coupling and a colorful version of the Friedman-Pippenger tree embedding technique, we establish a similar result in the setting of sparse random graphs, showing that there exists $C$ such that if the $G_i$ are independent random graphs sampled from $G(n,p)$, where $p\geq C\log n/n$, then $(G_1,\dots,G_n)$ is Hamilton-universal with high probability.
\end{abstract}

\section{Introduction}

One of the most fundamental questions in graph theory is that of determining whether a given graph contains a Hamilton cycle, i.e. a cycle which contains all vertices of the graph. There is no known efficient algorithm to check whether a graph contains such a cycle -- the problem of finding a Hamilton cycle in a given graph is famously NP-complete. Therefore, researchers focused on finding easy-to-check conditions which ensure that a graph contains such a cycle, and perhaps even allow for finding the cycle in polynomial time. The most famous condition of the kind is Dirac's condition~\cite{D52}, dating back to 1952, which states that every $n$-vertex graph of minimum degree at least $n/2$ must contain a Hamilton cycle. 

Although Dirac's condition is undoubtedly elegant, and tight in the sense that the degree 
requirement $n/2$ cannot be lowered, it is quite restrictive -- for example, it can only be satisfied by dense graphs which have quadratically many edges. It is, therefore, natural to ask what kind of conditions guarantee the existence of Hamilton cycles in sparser graphs. A common family of sparse graphs to study are the random graphs.

The study of Hamilton cycles in random graphs is almost as old as the study of random graphs -- in their original paper from 1960 \cite{ER60}, Erd\H{o}s and R\'enyi asked for the number of edges at which one expects to start seeing a Hamilton path in a random graph $G(n, m)$. This question was answered by P\'osa \cite{PO76}, who showed that one can expect Hamilton cycles in random graphs with $\Omega(n\log n)$ edges, by introducing a very useful and elegant technique now known as the P\'osa rotation. His argument was later tightened by Koml\'os and Szemer\'edi \cite{KS83}, and Bollob\'as \cite{B84} and Ajtai, Koml{\'o}s and Szemer{\'e}di \cite{AKS85}, to ultimately attain very good understanding of Hamilton cycles in random graphs. Namely, Bollob\'as, as well as Ajtai, Koml{\'o}s and Szemer{\'e}di, showed that if we order the edges of $K_n$ randomly and add them one by one to our graph, the time at which the graph becomes Hamiltonian with high probability coincides with the first time at which no vertex has degree less than 2, which is clearly a necessary condition for the graph to be Hamiltonian. For a detailed story of Hamilton cycles in random graphs, see the survey of Frieze \cite{friezesurvey}.

In recent years, the study of Hamilton cycles and other spanning structures in graphs has taken an interesting turn, by considering the colorful variants of classical questions from extremal graph theory. In particular, the setup we will be studying was introduced by Joos and Kim in \cite{JK20} while answering a question of Aharoni \cite{A20}, and it goes as follows. Given a collection of $m$ graphs $(G_1, \dots, G_m)$ on the same vertex set $V$, which we think of as colors of the edges (where an edge may receive an arbitrary number of colors), and a graph $H$ with vertices in $V$, we say that $(G_1, \dots, G_m)$ contains a \textit{rainbow (or transversal) copy} of $H$ if there exists an injective function $\psi:E(H)\to [m]$ such that the edge $e$ is contained in the graph $G_{\psi(e)}$ for all $e\in E(H)$. In other words, the graph $H$ can be constructed by picking at most one edge from each of the graphs $G_i$. Joos and Kim showed that if each of the graphs $G_i$ is an $n$-vertex graph with minimum degree $n/2$ and if $m\geq n$, then there exists a transversal Hamilton cycle in the collection $(G_1, \dots, G_m)$. 

This result inspired a large body of work on the subject, for example studying the existence of transversal cliques \cite{A20} and transversal $F$-factors \cite{Alp21}. We recommend the survey of Sun, Wang and Wei \cite{Survey24} to the interested reader who would like to learn more about the transversal subgraphs of graph collections. In this paper, we will consider the question of universality for Hamilton cycles in collections of random graphs.

The notion of universality was introduced by Bowtell, Morris, Pehova and Staden in \cite{BP24}. We say that a collection of graphs $(G_1, \dots, G_n)$ on the same vertex set of size $n$ is \textit{Hamilton-universal} if for each map $\chi:[n]\to[n]$, which we call the \textit{color-pattern}, there exists a Hamilton cycle whose $i$-th edge lies in $G_{\chi(i)}$. The authors of \cite{BP24} showed that every collection $(G_1, \dots, G_n)$ of $n$ graphs on the same vertex set of size $n$ satisfying $\delta(G_i)\geq (1/2+o(1))n$ is Hamilton-universal. Note that the difference from the result of Joos and Kim is that the edges of the Hamilton cycle must now come from a prescribed member of the collection $(G_1, \dots, G_n)$, rather than an arbitrary one. Extending \cite{BP24}, Heath, Hyde, Morrison and Ogden \cite{heath2025universality} recently studied universality of powers of Hamilton cycles.

The aim of this paper is to extend the result of Bowtell, Morris, Pehova and Staden to the setting of random graphs, answering a question of Pehova from the $30$th British Combinatorial Conference.

\begin{theorem}\label{thm:main}
There exists $C$ such that for $p\geq C\log n/n$ the following holds with high probability. Let $(G_1, \ldots, G_n)$ be a tuple of independent random graphs on the same vertex set sampled from $G(n,p)$. Then, $(G_1, \ldots, G_n)$ is Hamilton-universal. 
\end{theorem}
Note that the above result is tight up to the constant $C$, since by considering the color-pattern $\chi(i)=1$ for all $i$, the graph $G_1$ is required to contain a Hamilton cycle for $(G_1,\dots,G_n)$ to be Hamilton-universal. While showing the existence of a Hamilton cycle with any specific color-pattern is sufficient for Hamilton-universality in the deterministic setting, dealing with the looming union bound over all color-patterns is the key obstacle to overcome towards Theorem~\ref{thm:main}. 

\subsection{Related work}

While Hamilton-universality has not been studied before in the random setting of Theorem~\ref{thm:main}, there has been some interest in rainbow Hamilton cycles. Anastos and Chakraborti \cite{AC23} showed $(G_1,\ldots,G_n)$ with high probability contains a rainbow Hamilton cycle when each $G_i$ is sampled independently from $G(n,p)$ with $p= \Omega(\log n/n^2)$. For $p=\omega(\log n/n)$, a robust version of this statement has been obtained by Ferber, Han and Mao \cite{ferber2022dirac}. 

Colored Hamilton cycles have also been investigated in a different random model. Let us fix some number $k$ and let $\alpha$ be some distribution on the set of colors $[k]$. We denote by $G(n,p,\alpha)$ the colored random graph obtained by sampling $G(n,p)$ and then, coloring every edge independently according to $\alpha$. In a sequence of improvements, Cooper and Frieze \cite{cooper2002multi}, Frieze and Loh \cite{frieze2014rainbow} and Ferber and Krivelevich \cite{ferber2016rainbow} determined for which $p$ and $k$ such a graph contains a rainbow Hamilton cycle when $\alpha$ is the uniform distribution. For $k=O(1)$ and arbitrary $\alpha$, Frieze and Pegden \cite{frieze2025sequentially} determined the threshold for $p$ such that any specific coloring of a Hamilton cycle is present with high probability. The result in this direction which is closest to universality is given by Chakraborti, Frieze and Hasabnis \cite{CFH23}, who study not specific patterns but the set of vectors $(a_1,\ldots,a_k)$ such that there is a Hamilton cycle with exactly $a_i$ edges of color $i$. For more results in this direction, we recommend \cite[Section 2.6]{friezesurvey}.

One of the key tools in our proof is a colorful extension of the Friedman-Pippenger tree embedding technique. We would like to point out that a very similar setup was introduced independently by Lund and Xu in \cite{lund2025embedding} to embed edge-colored clique subdivisions in pseudorandom graphs. We thank the authors of \cite{lund2025embedding} for pointing out the overlap.
\subsection{Proof sketch}

We will now go through a short outline of the proof of our result. The tool underlying our proof, called McDiarmid's coupling, is a method which allows us to go from Hamilton cycles in a single random graph to rainbow Hamilton cycles in a collection of $n$ independent random graphs. Introduced by McDiarmid \cite{MD80} in 1980, the coupling was mostly overlooked in the random graph community until its recent revival by Ferber \cite{Fe15}, Ferber and Long \cite{FL19} and Montgomery \cite{Mo24}. 

The gist of the method is encapsulated in the following statement, which is essentially due to McDiarmid \cite{MD80}, albeit in a slightly different language. If $\cF$ is a family of ordered $n$-tuples with distinct elements from a ground set $E$ and $S_0, \dots, S_n$ are random subsets of $E$ including each element independently with probability $p$, then $S_1\times\dots\times S_n$ is at least as likely to contain an $n$-tuple of $\cF$ as $S_0\times \dots\times S_0$ is. In our setting, one should think of the ground set $E$ as the collection of all edges of the complete graph on $n$ vertices, thus making the random sets $S_0, \dots, S_n$ Erd\H{o}s-R\'enyi random graphs, while $\cF$ will be a collection of tuples of edges, roughly speaking, representing the Hamilton cycles of the complete graph. 
However, the actual statement we use is slightly more general, since we need a bit more flexibility. 

\begin{lemma}\label{lemma:McDiarmid coupling}
Let $n\geq 1$ be a positive integer and $E$ a set. Let $\cF$ be a family of ordered $n$-tuples from $E\cup\{\star\}$ containing each element of $E$ at most once. Let $S'_0, \dots, S'_n$ be i.i.d. random subsets of $E$, where the element $e\in E$ is included in $S'_i$ with probability $p_e\in [0, 1]$ and let $S_i = S'_i\cup\{\star\}$. 
Then, for any map $\chi:[n]\to [n]$, the probability that $S_0\times \dots\times S_0$ contains an $n$-tuple of $\cF$ is smaller or equal to the probability that $S_{\chi(1)}\times\dots\times S_{\chi(n)}$ contains an $n$-tuple of $\cF$.
\end{lemma}

Here is a simple application of Lemma~\ref{lemma:McDiarmid coupling}. Let $\cC$ denote the set of edge-ordered Hamilton cycles of the complete graph on $n$ vertices. That is $\cC$ is the set of $n$-tuples of edges $(e_1, \dots, e_n)$ forming a Hamilton cycle in this order. If we apply Lemma \ref{lemma:McDiarmid coupling} with $\cF=\cC$ and any fixed map $\chi:[n]\rightarrow [n]$, it follows that 
a collection of $n$ independent random graphs is at least as likely to contain a $\chi$-colored Hamilton cycle 
as a single random graph is to contain a Hamilton cycle. Therefore, a collection of $n$ independent random graphs $G(n, p)$ with high probability contains a Hamilton cycle in a given color-pattern already at probability $p\gg \log n/n$. 

Note that in order to prove Theorem~\ref{thm:main}, we need to find a Hamilton cycle for \textit{every} color-pattern. Therefore, we would like to apply a union bound over the set of all color-patterns. 
However, the random graph $G(n, C\log n/n)$ fails to have a Hamilton cycle with probability at least $n^{-C}$, since it has isolated vertices with at least this probability. Therefore, applying McDiarmid's coupling together with a simple union bound is not enough to complete the proof. 

The way to circumvent this issue and boost the probability of success is to show that with \textit{very} high probability, one can remove up to $\eps n$ vertices from $G(n,C\log n/n)$ so that the rest is Hamilton-connected, for some small constant $\eps>0$. Here, we say that a graph $G$ is Hamilton-connected if $G$ contains a Hamilton path with endpoints $u,v$ for any distinct vertices $u,v\in V(G)$. Before we state things more precisely, it will be very convenient for us to partition the set of vertices into $L$ and $R$ with $|L|=\lfloor n/2\rfloor$ and $|R|=\lceil n/2\rceil$. From now on, this partition will be fixed throughout the proof.

\begin{lemma}\label{lem:random graphs are hamilton-connected}
For every $\eps>0$ there exists a constant $C>0$ such that the following holds. If $G\sim G(n, p)$ is a random graph where $p\geq C\log n/n$, with probability at least $1-n^{-2n}$ we have the following. There exists some set $X$ of at most $\varepsilon n$ vertices such that for every set $Y\subseteq L$, the induced subgraph of $G$ on $V(G)\backslash(X\cup Y)$ is Hamilton-connected.
\end{lemma}

This lemma gives us not only a way of finding almost-spanning paths in $G$, but we are also allowed to pick and choose which vertices of $L$ we want to include in our path -- something we will find very useful later on. However, the almost-spanning path we construct will not be able to cover vertices of $X$. Therefore, we need a final lemma, which allows us to cover sets of at most $\eps n$ vertices using a path with prescribed edges colors. For that, we need a different argument based on a colorful version of the Friedman-Pippenger tree embedding technique, which was first introduced in \cite{FP87} and further developed by Glebov, Johanssen and Krivelevich \cite{glebov2013hamilton}, Montgomery \cite{montgomery2019spanning} and Dragani{\'c}, Krivelevich and Nenadov \cite{draganic2022rolling}. 

Given a tuple $(G_1, \ldots, G_n)$ of graphs and a map $\chi:[n]\to[n]$, we say that $(e_1, \dots, e_n)\in \cC$  is a \emph{$\chi$-colored Hamilton cycle} if $e_i\in E(G_{\chi(i)})$ for all $1\leq i\leq n$ and, similarly, a (not necessarily Hamiltonian) path $(e_1, \dots e_k)$ is \textit{$\chi$-colored} if $e_i\in E(G_{\chi(i)})$ for all $1\leq i \leq k$.

\begin{lemma}\label{lem: absorption}
There exists $C$ such that for any $\eps\in (0, 10^{-4})$, large integer $n$ and $p\geq C\log n/n$ the following holds with high probability. Let $(G_1, \ldots, G_n)$ be a tuple of independent random graphs on the same vertex set sampled from $G(n,p)$. Then, for every map $\chi:[n]\to [n]$ and every set $X$ of at most $\varepsilon n$ vertices, there exists a $\chi$-colored path $P$ with $X\subseteq V(P)\subseteq X\cup L$ such that both endpoints of $P$ are outside $X$.
\end{lemma}

Let us now show how one can complete the proof, modulo these three lemmas.

\begin{proof}[Proof of Theorem~\ref{thm:main}, assuming Lemmas~\ref{lemma:McDiarmid coupling}, \ref{lem:random graphs are hamilton-connected}, \ref{lem: absorption}.]
Let us set $\eps=10^{-5}$ and let $C$ be a sufficiently large constant so that Lemma~\ref{lem:random graphs are hamilton-connected} applies. We begin the proof by a slight change in perspective. Instead of sampling the graphs from $G(n, p)$, for each $i\in [n]$ we generate two random graphs $G_i, H_i\sim G(n, p')$ with $p'=1-\sqrt{1-p}\geq p/2\geq C\log n/2n$, independently for each $i$. Then, the graph $G_i\cup H_i$ follows the distribution of $G(n, p)$. Hence, to prove the statement of the theorem, it suffices to show that the $n$-tuple of graphs $(G_1\cup H_1, \dots, G_n\cup H_n)$ is Hamilton-universal with high probability. 

By Lemma~\ref{lem: absorption}, the graphs $(G_1, \dots, G_n)$ with high probability have the property that for each map $\chi:[n]\to [n]$, and for each set of vertices $X\subseteq [n]$ with $|X|\leq \eps n$, there is a $\chi$-colored path $P_X$ in $(G_1, \dots, G_n)$ which covers $X$ and is disjoint from $R\setminus X$. Throughout the rest of the proof, fix a single outcome of $(G_1, \dots, G_n)$ and condition on the fact this outcome satisfies the conclusion of Lemma~\ref{lem: absorption}. 

Having fixed the graphs $(G_1,\dots,G_n)$, we next analyze the probability that $(G_1\cup H_1, \dots G_n\cup H_n)$ contains a $\chi$-colored Hamilton cycle for some given color-pattern $\chi:[n]\rightarrow [n]$, with the intention of using the union bound over all $\chi$. Here, the probability refers to the choice of the graphs $(H_1, \dots, H_n)$, since the $n$-tuple $(G_1, \dots, G_n)$ has been fixed. For any such $\chi$, let $\mathbf{A}_\chi=\mathbf{A}_\chi(G_1,\dots,G_n)$ denote the set of $n$-tuples of graphs $(H_1,\dots, H_n)$ such that $(G_1\cup H_1, \dots, G_n\cup H_n)$ contains a $\chi$-colored Hamilton cycle. 
Let $H_0$ denote an additional random graph from $G(n, p')$. We aim to show the following two statements.

\begin{enumerate}[label=(\alph*)]
    \item\label{(a)} $\Pr[(H_0, \ldots, H_0) \in \mathbf{A}_\chi]\geq 1-n^{-2n}$.
    \item\label{(b)} $\Pr[(H_1,\ldots,H_n)\in\mathbf{A}_\chi]\geq \Pr[(H_0,\ldots,H_0)\in \mathbf{A}_\chi]$.
\end{enumerate}

First, we prove \ref{(a)} using Lemma~\ref{lem:random graphs are hamilton-connected}. Applying this lemma to $H_0$, if $C$ is large enough, with probability at least $1-n^{-2n}$ there exists a set $X\subseteq V$ with $|X|\leq \varepsilon n$ such that for every $Y\subseteq L$, the graph $H_0-X-Y$ is Hamilton-connected. If such a set $X$ exists and $C$ is large enough, recall that $(G_1, \dots, G_n)$ satisfies the conclusion of Lemma~\ref{lem: absorption}, meaning that there is a $\chi$-colored path $P=e_1e_2\dots e_k$ with $X\subseteq V(P)\subseteq X\cup L$ and with endpoints outside $X$. If we denote the endpoints of $P$ by $u, v$, the graph $H_0-(V(P)\backslash\{u, v\})$ is Hamilton-connected and, in particular, contains a Hamilton path $Q=e_{k+1}\dots e_n$ from $v$ to $u$. Hence $(G_1\cup H_0, \dots, G_n\cup H_0)$ contains the $\chi$-colored Hamilton cycle $PQ=e_1\dots e_n$, which implies $(H_0,\dots,H_0)\in \mathbf{A}_\chi.$

To prove \ref{(b)}, we apply Lemma~\ref{lemma:McDiarmid coupling}. Let $E=E(K_n)$ and define the collection of $n$-tuples $\cF$ of elements in $E$ as follows: for each $n$-tuple of edges forming a Hamilton cycle $(e_1,\dots,e_n)\in  \cC$, we define a corresponding tuple $(f_1, \dots, f_n)$, where $f_i=e_i$ if $e_i\notin E(G_{\chi(i)})$ and $f_i=\star$ otherwise. Let $\mathcal{F}$ denote the set of all such $(f_1,\dots,f_n)$. Observe that, since $H_0, \dots, H_n$ are sampled from $G(n, p')$, their edge-sets $E(H_i)$ are random subsets of $E$, including each element $e\in E$ with probability $p_e=p'$. 

It follows from the above definitions that $(G_1\cup H_1, \dots, G_n\cup H_n)$ contains a $\chi$-colored Hamilton cycle, that is $(H_1, \dots, H_n)\in\mathbf{A}_\chi$, if and only if 
\[(f_1, \dots, f_n)\in \big(E(H_{\chi(1)})\cup\{\star\}\big)\times \dots \big(E(H_{\chi(n)})\cup\{\star\}\big)\] 
for some $(f_1,\dots,f_n)\in \cF$. Similarly $(H_0,\dots, H_0)\in \mathbf{A}_\chi$ if and only if 
\[(f_{1}, \dots, f_{n})\in \big(E(H_0)\cup\{\star\}\big)\times \dots \big(E(H_0)\cup\{\star\}\big)\] for some $(f_1,\dots,f_n)\in \cF$. The conclusion of Lemma~\ref{lemma:McDiarmid coupling} then directly implies \ref{(b)}.

Having confirmed inequalities \ref{(a)} and \ref{(b)}, Theorem \ref{thm:main} follows from a union bound argument over the possible color-patterns $\chi:[n]\rightarrow [n]$. Namely, if $(G_1, \dots, G_n)$ satisfies the conclusion of Lemma~\ref{lem: absorption}, which happens with high probability, we have
\begin{align*}
    &\Pr\Big[(G_1\cup H_1, \dots, G_n\cup H_n)\text{ is not Hamilton-universal}\Big| (G_1,\dots,G_n)\Big]\leq \\
    &\qquad\sum_{\chi:[n]\to [n]}\Pr\Big[(H_1, \dots, H_n)\not\in \mathbf{A}_\chi\Big]\leq n^n\cdot n^{-2n}=o(1).
\end{align*}
Therefore, we conclude that a collection of $n$ independent random graphs from $G(n, p)$ is Hamilton-universal with high probability when $p\geq C\log n/n$.
\end{proof}

\section{Proof details}
\subsection{McDiarmid's coupling}

The proof of Lemma \ref{lemma:McDiarmid coupling} follows the strategy of \cite{MD80}. We include the proof for the convenience of the reader since our setup is slightly different than that of \cite{MD80}.

\begin{proof}[Proof of Lemma~\ref{lemma:McDiarmid coupling}.]
Let $E=\{e_1,\dots, e_m\}$ and let $\cF\subseteq (E\cup \{\star\})^n$ be such that each $n$-tuple in $\cF$ contains each element $e_i\in E$ at most once. We will prove the statement by considering a sequence $\cA_0,\ldots,\cA_m$ of $n$-tuples of subsets of $E$.
For every $0\leq i\leq m$, let $\cA_i=(A_1^i,\ldots,A_n^i)$ be as follows:
\begin{itemize}
   \item for every $1\leq j\leq i$ and $1\leq k\leq n$, we include $e_j$ in $A_k^i$ if $e_j\in S'_{\chi(k)}$;
   \item for every $i+1\leq j\leq m$ and $1\leq k\leq n$, we include $e_j$ in $A_k^i$ if $e_j\in S'_{0}$.
\end{itemize}
By definition of these $n$-tuples, we have that $\cA_0=(S'_0, \dots, S'_0)$ and $\cA_m=(S'_{\chi(1)}, \dots, S'_{\chi(n)})$. Therefore, it suffices to prove that, for every $1\leq i\leq m$, $(A^{i-1}_1\cup\{\star\})\times\dots\times (A^{i-1}_n\cup\{\star\})$ is at most as likely to contain a tuple from $\cF$ as $(A^i_1\cup\{\star\})\times \dots\times (A^i_n\cup\{\star\})$. Observe that $A^{i-1}_k$ and $A^{i}_k$ only potentially differ on $e_i$ for each $1\leq k \leq n$. Let $T'_k=A^{i-1}_k\backslash \{e_i\}=A^{i}_k\backslash \{e_i\}$ and $T_k = T'_k\cup\{\star\}$.

Let us consider the conditional probability for either $(A^{i-1}_1\cup\{\star\})\times \dots\times (A^{i-1}_n\cup\{\star\})$ or $(A^{i}_1\cup\{\star\})\times \dots\times (A^{i}_n\cup\{\star\})$ to contain an $n$-tuple of $\cF$, given the sets $T_1,\dots,T_n$. Let $P_{i-1}=P_{i-1}(T_1,\dots,T_n)$ and $P_i=P_i(T_1,\dots,T_n)$ denote these respective conditional probabilities. We have three potential outcomes:
\begin{itemize}
    \item either $T_1\times\dots\times T_n$ already contains an $n$-tuple of $\cF$, or
    \item $(T_1\cup \{e_i\})\times\dots\times (T_n\cup \{e_i\})$ contains an $n$-tuple $\cF$ but $T_1\times\dots\times T_n$ does not, or
    \item $(T_1\cup \{e_i\})\times\dots\times (T_n\cup \{e_i\})$ does not contain an $n$-tuple of $\cF$.
\end{itemize}

In the first case, $(A^{i-1}_1\cup\{\star\})\times \dots\times (A^{i-1}_n\cup\{\star\})$ and $(A^{i}_1\cup\{\star\})\times \dots\times (A^{i}_n\cup\{\star\})$ both contain an $n$-tuple of $\cF$ with probability $1$, so $P_{i-1}=P_i=1$.

In the second case, let $(f_1, \dots, f_n)\in \cF$ be an $n$-tuple contained in $(T_1\cup \{e_i\})\times\dots\times (T_n\cup \{e_i\})$. Then, $(A^{i-1}_1\cup\{\star\})\times \dots\times (A^{i-1}_n\cup\{\star\})$ contains an $n$-tuple of $\cF$ with probability $p_{e_i}$, since this is the probability $e_i$ is included in $S_0'$ and thus, in the sets $A^{i-1}_1, \dots, A^{i-1}_n$. Since $T_1\times\dots\times T_n$ contains no $n$-tuple from $\cF$, $e_i$ appears in $(f_1, \dots, f_n)$ on the $\ell$-th coordinate, say. As $e_i$ appears at most once in $(f_1, \dots, f_n)$, the other elements of $(f_1, \dots, f_n)$ are already present in the appropriate sets of the tuple $(T_1, \dots, T_n)$.
Then, $(f_1, \dots, f_n)$ is completed in $(A^{i}_1, \dots, A^{i}_n)$ if $e_i$ is included in $A^{i}_{\ell}$, i.e. if $e_i\in S'_{\chi(\ell)}$, which happens with probability $p_{e_i}$. Hence, we conclude that $P_{i}\geq p_{e_i}=P_{i-1}$.

Finally, in the third case, clearly $P_{i-1}=P_{i}=0$. Thus, we have in all cases that $P_{i-1}\leq P_{i}$, completing the proof.
\end{proof}
\subsection{Almost spanning paths in random graphs}

We will next prove Lemma~\ref{lem:random graphs are hamilton-connected}.
In doing so, we use the following result about finding Hamilton paths in expanders, which was proven in \cite[Theorem 7.1.]{draganic2024hamiltonicity}.

\begin{theorem} \label{thm:expanders are Hamilton-connected}
There exists an absolute constant $C\geq 1$ such that the following statement holds. Let $G$ be a graph on $n$ vertices with the following two properties
\begin{itemize}
    \item  $|N(U)|\geq C|U|$ for all sets $U\subseteq V(G)$ of size $|U|< n/2C$, and
    \item there exists an edge between any two disjoint sets $U, W\subseteq V(G)$ of size $|U|, |W|\geq n/2C$.
\end{itemize}
Then, $G$ is Hamilton-connected. In other words, for any two distinct vertices $x, y\in V(G)$ there exists a Hamilton path in $G$ with these two endpoints.
\end{theorem}

In the literature, the graphs satisfying the two properties mentioned above are often called $C$-expanders. However, since we also consider different kinds of expanders in this paper, we prefer to avoid reusing the terminology.

Let us now give the proof of Lemma~\ref{lem:random graphs are hamilton-connected}.

\begin{proof}[Proof of Lemma~\ref{lem:random graphs are hamilton-connected}.]
As indicated before, the key ingredient of the proof is Theorem~\ref{thm:expanders are Hamilton-connected}, which ultimately ensures the Hamilton-connectedness of the induced subgraph on $V(G)\backslash (X\cup Y)$. Hence, let $C_0$ be the absolute constant from Theorem~\ref{thm:expanders are Hamilton-connected} and set $K=\max\{1/\eps, C_0\}$. Finally, we set $C=200 K^2$.

Let us denote by $\cA$ the event that the random graph $G\sim G(n, p)$ has the property that any two disjoint sets $U, W\subseteq V(G)$ of size $\frac{n}{10K}$ have an edge between them. To bound the probability of the complement $\overline{\cA}$, we use a union bound over all pairs of disjoint sets $(U, W)$ of size at least $\frac{n}{10K}$.

The probability that no edge exists between $U$ and $W$ is $(1-p)^{|U||W|}\leq e^{-p\cdot (n/10K)^2}$. Recall that $p\geq C\log n/n$. Since there are at most $3^n$ pairs of disjoint sets $U, W\subseteq V(G)$, and by the choice of the constant $C$, we conclude that
    \begin{align*}
        \Pr\left[\hspace{0.03cm}\overline{\mathcal{A}}\hspace{0.03cm}\right]\leq 3^n e^{-\frac{pn^2}{100K^2}}\leq n^{-2n}.
    \end{align*}

Suppose now that $\cA$ holds and let us define the set $X$ to be the largest set of size $|X|\leq \frac{n}{10K}\leq \eps n$ which has $|N_G(X)\cap R|\leq K|X|$. What remains to show is that for every set $Y\subseteq L$, the subgraph of $G$ induced on $V(G)\backslash (X\cup Y)$ satisfies the assumptions of Theorem~\ref{thm:expanders are Hamilton-connected}. Once this is done, we will immediately obtain the conclusion that this subgraph is Hamilton-connected, as needed.

Hence, let us fix a set $Y\subseteq L$ and let us denote the number of vertices of $G'=G[V(G)\backslash (X\cup Y)]$ by $m\geq n-|X|-|Y|\geq 4n/10$. Observe that for any two sets $U, W\subseteq V(G')$ of size $|U|, |W|\geq \frac{m}{2K}\geq \frac{n}{5K}$, there must be an edge between them since $\cA$ holds. Hence, we only need to verify that for any $U\subseteq V(G')$ of size $|U|\leq \frac{m}{2K}$ one has $|N_{G'}(U)|\geq K|U|$.

We have three cases. If $|U\cup X|\leq \frac{n}{10 K}$, then by maximality of $X$ we have $|N_G(U\cup X)\cap R|\geq K(|U|+|X|)$. Since $X$ has at most $K|X|$ neighbors in $R$ by assumption, we conclude $|N_G(U)\cap R\setminus X|\geq K|U|$ and in particular $|N_{G'}(U)|\geq K|U|$ since $R\setminus X\subseteq V(G')$.

On the other hand, if $|U\cup X|\geq \frac{n}{10K}$ and $|U|\leq \frac{n}{10K}$, there are at most $\frac{n}{10 K}$ vertices from $R$ with no neighbors in $U\cup X$, since the event $\cA$ holds. Hence, $|N_{G}(U\cup X)\cap R|\geq \frac{n}{2}-\frac{n}{10K}\geq \frac{4n}{10}$ and since $X$ has at most $K|X|\leq \frac{n}{10}$ neighbors in $R$, we find $|N_{G}(U)\cap R|\geq \frac{3n}{10}$, which is enough since $K|U|\leq \frac{n}{10}$.

Finally, if $|U|\geq \frac{n}{10K}$, then there are at most $\frac{n}{10K}$ vertices in the whole graph non-adjacent to $U$. In particular, this means that $U$ has at least $m-|U|-\frac{n}{10K}>m/2$ neighbors in $G'$, which completes the proof.
\end{proof}

\subsection{Covering small sets using Friedman-Pippenger embeddings}

In this section, we give the proof of the Lemma~\ref{lem: absorption}. To recall, this lemma states that if $(G_1, \dots, G_n)$ is a collection of graphs sampled independently from $G(n, C\log n/n)$, with the same vertex set $V=L\cup R$, then with high probability for every color-pattern $\chi:[n]\to [n]$ and every set $X\subseteq V$ of size $|X|\leq \eps n$, one can cover $X$ with a $\chi$-colored path $P$ satisfying $V(P)\subseteq L\cup X$. 

Our proof will proceed in two stages. In the first stage, we will show that with high probability, for every color-pattern $\chi:[n]\to [n]$ and every set $X\subseteq V$ of size at most $\eps n$, one can cover all but $\eps n/\log n$ vertices of $X$ using a path that alternates between $X$ and $L\backslash X$. After this stage, one is left with only $\eps n/\log n$ vertices which need to be incorporated into a path, which can be done using a Friedman-Pippenger embedding technique. The main difficulty of this technique is that we work in a colorful setting, meaning that we need to carefully redefine the notions of expansion and good embeddings which are usually used. 

Let us begin by setting up the definitions we need to make the Friedman-Pippenger machinery work and proving variants of several standard lemmas adapted to our setting. As mentioned in the introduction, a very similar setup has been considered independently by Lund and Xu \cite{lund2025embedding}. In particular, Definition 2.2, Proposition 2.3 and Lemma 2.4 in \cite{lund2025embedding} correspond to equation (\ref{eqn:goodness}), Claim~\ref{claim:indicator_identity} and Lemma~\ref{lem: backward} in our paper.

Let $(G_1, \dots G_n)$ be an $n$-tuple of graphs on the vertex set $V$, and let us refer to the edges of the graph $G_i$ as the edges of \textit{color} $i$. A pair of a vertex $v\in V$ and a color $i\in [n]$ is called a \textit{color-vertex}, and a \textit{color-set} $A\subseteq V\times [n]$ is just a set of color-vertices.

For a color-set $A$ and integers $m, D>0$, we say that the collection $(G_1, \ldots, G_n)$ is an \textit{$(m, D, A)$-color expander} if for all $S\subseteq A$ of size at most $m$, we have $|N(S)|\geq D|S|$. Here, $N(S)=\bigcup_{(u, i)\in S} N_{G_i}(u)$ stands for the union of the neighbors of $u$ in $G_i$, where $(u, i)$ ranges over $S$. 

This definition perhaps needs a bit of motivation. One should think of $A$ as the set of color-vertices which ``expand" well. This set $A$ will be defined formally as soon as the proof of Lemma~\ref{lem: absorption} starts, and once it is defined it will not change throughout the proof. Intuitively we should think of $A$ as those pairs $(u, i)$ where the vertex $u$ has many neighbors in color $i$ outside $X$. Typically, this will be almost all color vertices, so one can intuitively think of $A$ as the set of all color vertices.

Finally, in order to show that edge-colored forests can be embedded into $(G_1, \dots, G_n)$, we need the notion of a good embedding of a forest $F$. Hence, let us start from an edge-colored forest $F$, with the set of colors $[n]$. For bookkeeping purposes, we will also assume the edges of our forest are oriented, but this assumption is only technical.
We say that an injective function $\varphi:V(F)\to V$ is an \textit{embedding} if $\varphi(u)$ and $\varphi(v)$ are adjacent in the graph $G_i$, whenever $uv$ is an edge in $F$ of color $i$. For an embedding $\varphi$, we denote by $\varphi(F)$ the image of $V(F)$ under $\varphi$.

Further, this embedding is called \textit{$(m, D, A)$-good} if for every $S\subseteq A$ with $|S|\leq m$, it holds that
 \begin{align}\label{eqn:goodness}
     |N(S)\setminus \varphi(F)|\geq &\Big|S\cap \big\{(\varphi(u), i)\,|\,u\in V(F)\text{ has an in-edge of color $i$}\big\}\Big|\nonumber\\
     +&\sum_{(x, i)\in S} D-(\text{number of edges of color $i$ incident to $\varphi^{-1}(x)$ in $F$}).
 \end{align}

Finally, if we have a set of edges of $G_1, \dots, G_n$ which can be obtained as the set of edges of some forest under an $(m, D, A)$-good embedding, we will say that this set of edges is \textit{$(m, D, A)$-good}. Correspondingly, one can say that an edge-colored subgraph of $G_1, \dots, G_n$ is \textit{$(m, D, A)$-good} under the same condition.

Let us begin by showing that good embeddings can be extended.

\begin{lemma}[Extension Lemma]\label{lem: forward}
Let $F$ be an oriented edge-colored forest, with at most $mD$ vertices, maximum degree $D$, with a leaf $w$ and an edge $v\to w$ of color $\ell$, and let $K=F-w$ be the forest obtained by removing the leaf $w$ from $F$. 
Further, let $(G_1, \ldots, G_n)$ be a $(2m, 2D+1, A)$-color expander. If $(v, \ell)\in A$, then each $(m, D, A)$-good embedding of $K$ into $(G_1, \dots, G_n)$ can be extended to an $(m, D, A)$-good embedding of $F$.
\end{lemma}
\begin{proof}
Let $\varphi$ be the $(m, D, A)$-good embedding of $K$ into $(G_1, \dots, G_n)$. We will try to construct the embedding $\psi$ of $F$ by extending $\varphi$, i.e. by defining $\varphi(w)$ to be one of the neighbors of $\varphi(v)$ in $G_\ell$. Let us denote the neighbors of $\varphi(v)$ in $G_\ell$ which are not already used by the embedding $\varphi$ by $w_1, \dots, w_k$, i.e. set $\{w_1, \dots, w_k\}=N_{G_\ell}(\varphi(v))\backslash \varphi(K)$. 

To show that the set $N_{G_\ell}(\varphi(v))\backslash \varphi(K)$ is nonempty, we recall that $\varphi$ is a good embedding, and apply the corresponding condition to the set $S=\{(v, \ell)\}$ to get $|N_{G_\ell}(\varphi(v))\backslash \varphi(K)|\geq D-(\text{number of edges of color $\ell$ incident to $v$ in $K$})\geq 1$, since $v$ has degree at most $D-1$ in $K$.

Let $\phi_j$ be the embedding of $F$ which extends $\varphi$ by defining $\psi_j(v)=w_j$, for $1\leq j\leq k$. If there is some $j$ for which this is an $(m, D, A)$-good embedding, we are done. Otherwise, for each $\psi_j$, there is some set $S_j$ violating the $(m, D, A)$-goodness condition.

To obtain the contradiction, we now introduce the key quantity in the proof, associated to every color set $S\subseteq V\times [n]$
\begin{align*}
    f_\varphi(S) = |N(S)\setminus \varphi(K)|- &\Big|S\cap \big\{(\varphi(u), i)\,|\,u\in V(K)\text{ has an in-edge of color $i$}\big\}\Big|\\
     -&\sum_{(x, i)\in S} D-(\text{number of edges of color $i$ incident to $\varphi^{-1}(x)$ in $K$}).
\end{align*}

The basic observation is that $f_\varphi(S)\geq  0$ for all $|S|\leq m$ since $\varphi$ is an $(m, D, A)$-good embedding. The second observation is that $f_\varphi$ is a submodular function.

\begin{claim}
The function $f_\varphi$ is submodular, i.e. for all $S,T\subseteq A$, it holds that $f_\varphi(S\cup T)+f_\varphi(S\cap T)\leq f_\varphi(S) + f_\varphi(T)$.
\end{claim}
\begin{proof}
Note that the second and third summand in the definition of $f_\varphi$ can be rewritten as a sum over the elements of $S$. Therefore, these two summands are modular functions and we only need to show that the summand $|N(S)\setminus \varphi(K)|$ is submodular. In other words, we need to show that $|N(S\cap T)\setminus \varphi(K)|+|N(S\cup T)\setminus \varphi(K)|\leq |N(S)\setminus \varphi(K)|+|N(T)\setminus \varphi(K)|$. Note that for each $v\in N(S\cup T)\setminus \varphi(K)$, we must have $v\in (N(S)\cup N(T))\setminus \varphi(K)$. If we also have $v\in N(S\cap T)\setminus \varphi(F)$, then we have both $v\in N(S)\setminus \varphi(K)$ and $v\in N(T)\setminus \varphi(K)$. This shows submodularity since 
\begin{align*}
    |N(S\cup T)\setminus \varphi(K)|+|N(S\cap T)\setminus \varphi(K)|&\leq |N(S)\cup N(T)\setminus \varphi(K)|+|N(S)\cap N(T)\setminus \varphi(K)|\\
    &= |N(S)\setminus \varphi(K)|+|N(T)\setminus \varphi(K)|. 
\end{align*}
\claimproofend 

Having established submodularity of $f_\varphi$, we are ready to proceed with the proof. Our next goal is to understand how $f_\varphi$ changes when $\varphi$ is extended to $\psi_j$. 

\begin{claim}\label{claim:indicator_identity}
For each $S\subseteq A$, we have $f_{\psi_j}(S)=f_\varphi(S)+\mathbf{1}_{w_j\in N(S)}-\mathbf{1}_{(\varphi(v),\ell)\in S}$, where $\mathbf{1}_E$ denotes the indicator function for the condition $E$.
\end{claim}
\begin{proof}
Let us consider how the three summands in the definition of $f_\varphi$ change when $\varphi$ is extended. First, we have that for each $S\subseteq A$
    \begin{align*}
        |N(S)\setminus \varphi(K)| = |N(S)\setminus \psi_j(F)|+\mathbf{1}_{w_j\in N(S)}.
    \end{align*}
    Also, we observe
    \begin{align*}
        \Big|S\cap \big\{(\varphi(u), i)\,|\,u\in &V(K)\text{ has an in-edge of color $i$}\big\}\Big| \\
        &= \Big|S\cap \big\{(\psi_j(u), i)\,|\,u\in V(F)\text{ has an in-edge of color $i$}\big\}\Big|-\mathbf{1}_{(w_j,\ell)\in S}.
    \end{align*}
    Finally, the third summand is changed as follows
    \begin{align*}
        \sum_{(x,i)\in S}&D-(\text{number of edges of color $i$ incident to $\varphi^{-1}(x)$ in $K$}) \\
        &= \sum_{(x,i)\in S} D-(\text{number of edges of color $i$ incident to $\psi_j^{-1}(x)$ in $F$}) + \mathbf{1}_{(\varphi(v),\ell)\in S}+\mathbf{1}_{(w_j,\ell)\in S}.
    \end{align*}
Summing these three equations gives the proof of the claim.
\claimproofend
    
Let us now recall what we were doing in the proof. We assumed that none of the embeddings $\psi_j$ were an $(m,D,A)$-good embedding, which meant that we had sets $S_j\subseteq A$ witnessing that $\psi_j$ is not good, i.e. having $f_{\psi_j}(S_j)<0$. Since $f_{\psi_j}(S_j)=f_\varphi(S_j)+\mathbf{1}_{w_j\in N(S_j)}-\mathbf{1}_{(\varphi(v),\ell)\in S_j}$ and $f_\varphi(S_j)\geq 0$, it follows that $w_j\in N(S_j)$, $(\varphi(v),\ell)\notin S_j$ and $f_\varphi(S_j)= 0$.

The final observation is that for any $S\subseteq A$ of size $m\leq |S|\leq 2m$, we must have $f_\varphi(S)> 0$, since $(G_1, \dots, G_n)$ is a $(2m, 2D+1, A)$-color-expander implying that $|N(S)\setminus \varphi(K)| > (2D+1)|S|-mD$ and 
\begin{align*}
    \Big|S\cap \big\{(\varphi(u), i)\,|\,u\in V(K)\text{ has an in-edge of color $i$}\big\}\Big|&\leq |S|,\\
    \sum_{(x,i)\in S}D-(\text{number of edges of color $i$ incident to $\varphi^{-1}(x)$ in $F$})&\leq D|S|.
\end{align*}
This gives us $|N(S)\setminus \varphi(V(K))| > (2D+1)|S|-mD\geq (D+1) |S|$ and so $f_\varphi(S)> 0$.

Consider now the union $T_j=\bigcup_{t=1}^j S_t$. By submodularity of $f_\varphi$, we must have $f_\varphi(T_j)\leq 0$. Moreover, by induction on $j$ it is not hard to show $|T_j|< m$, since if we ever had $m\leq |T_j|\leq 2m$ we would get $f_\varphi(T_j)>0$, leading to a contradiction. Hence, the union of all sets $S_1, \dots, S_k$, denoted by $T_k$, has $f_\varphi(T_k)=0$ and $|T_k|< m$. But we showed that $\{w_1,\ldots w_k\}\subseteq N(T_k)$ and $(\varphi(v),\ell)\notin T_k$. It follows that $f_\varphi(T_k\cup \{(\varphi(v),\ell)\})<0$, contradicting that $\varphi$ is an $(m,D,A)$-good embedding. This completes the proof.
\end{proof}

The next lemma shows that $\varphi$ is a good embedding of a forest $F$, removing one leaf from this forest will leave this embedding good. 

\begin{lemma}[Rollback lemma]\label{lem: backward}
Let $F$ be an oriented edge-colored forest, with a leaf $w$ and an edge $v\to w$ of color $\ell$, and let $K=F-w$ be the forest $F$ with the leaf $w$ removed. Also, let $(G_1, \dots, G_n)$ be an $n$-tuple of graphs on the vertex set $V$.
Then, for every $(m, D, A)$-good embedding $\varphi:V(F)\to V$, the restriction $\varphi|_{K}:V(K)\to V$ is also an $(m, D, A)$-good embedding.
\end{lemma}
\begin{proof}
Our goal is to verify the $(m, D, A)$-goodness condition for the embedding $\varphi|_K$ by checking the inequality for each color-set $S\subseteq A$ with $|S|\leq m$. First of all, we have
\begin{align*}
     |N(S)\setminus \varphi|_K(K)| = |N(S)\setminus \varphi(F)| + \mathbf{1}_{\varphi(w)\in N(S)}.
\end{align*}
Then, we also have 
\begin{align*}
    \Big|S\cap \big\{(\varphi(u), i)\,|\,u\in &V(K)\text{ has an in-edge of color $i$}\big\}\Big| \\
    = \Big|S\cap &\big\{(\varphi(u), i)\,|\,u\in V(F)\text{ has an in-edge of color $i$}\big\}\Big|-\mathbf{1}_{(\varphi(w),\ell)\in S}.
\end{align*}
Finally, we have
    \begin{align*}
\sum_{(x,i)\in S} &D-(\text{number of edges of color $i$ incident to $\varphi|_K^{-1}(x)$ in $K$})\\= &\sum_{(x,i)\in S} D-(\text{number of edges of color $i$ incident to $\varphi^{-1}(x)$ in $F$}) + \mathbf{1}_{(\varphi(v),\ell)\in S} + \mathbf{1}_{(\varphi(w),\ell)\in S}.
    \end{align*}
    
Since the embedding $\varphi$ is $(m,D, A)$-good, we have 
 \begin{align*}
     |N(S)\setminus \varphi(F)|\geq &\Big|S\cap \big\{(\varphi(u), i)\,|\,u\in V(F)\text{ has an incoming edge of color $i$}\big\}\Big|\\
     +&\sum_{(x, i)\in S} D-(\text{number of edges of color $i$ incident to $\varphi^{-1}(x)$ in $F$}).
 \end{align*}
When we restrict $\varphi$ to $\varphi|_K$, the left-hand side increases by $\mathbf{1}_{\varphi(w)\in N(S)}$, while the right-hand side increases by $\mathbf{1}_{(\varphi(v),\ell)\in S}$. Hence, to establish the corresponding inequality for $\varphi|_K$, it suffices to show $\mathbf{1}_{\varphi(w)\in N(S)}\geq \mathbf{1}_{(\varphi(v),\ell)\in S}$, which is true since $\varphi(w)$ is a neighbor of $\varphi(v)$ in color $\ell$, due to $\varphi$ being an embedding.
\end{proof}

We have now established all the results we need about good embeddings, and it is time to use them to proof of Lemma~\ref{lem: absorption}.

The key object in the proof will be a star-system. A \textit{star-system} $M$ is a union of monochromatic stars and an associated color for each star (which may be different from the color of the edges of the star). More precisely we have the set of centers of stars, denoted by $U_M$, together with two assigned colors for each vertex $u\in U_M$ denoted by $i_u, j_u$, and a star of size $\log n$ centered at $u$ in the graph $G_{i_u}$. We refer to the color $j_u$ as the color assigned to this star, and note that we may have $i_u\neq j_u$. We denote the neighbors of $u$ in this star by $N_u$, and we assume that the sets $N_u, N_v$ are disjoint for $u\neq v$, and also disjoint from $U_M$. The \textit{size} of a star-system $M$ is simply the size of $U_M$ and the set of its vertices is $V(M)=U_M\cup_{u\in U_M} N_u$.

\begin{proof}[Proof of Lemma~\ref{lem: absorption}.]
The first step of the proof is to state three expansion properties of the random graphs that we need, and show these properties occur with high probability. After this, the proof is completed deterministically. Here are the properties we will need.

\begin{enumerate}[label=(\Alph*)]
       \item\label{itm: DFS-expansion} Let $M$ be a star-system, let $Y$ be a set disjoint from $V(M)$, and assume $|U_M|, |Y|\geq \eps n/\log n$. Then, there is some vertex $u\in U_M$ for which there is an edge of color $j_u$ between $N_u$ and $Y$.
       \item\label{itm: non-expanding set} For every $X\subseteq V$ with $|X|\leq \varepsilon n$ and color-set $W$ such that $\big|\big(V(W)\cup N(W)\big)\cap (L \setminus X)\big|\leq n/3$, it holds that $|W|\leq \varepsilon n/\log n$.
       \item\label{itm:joined graph} For any $i\in [n]$ and any sets $S, T\subseteq V$ of size $|S|=n/2^7, |T|=n/2^7$, there exists an edge between $S$ and $T$ in the graph $G_i$.
\end{enumerate}

\noindent
\textbf{Showing that (A) holds with high probability.} To show (A) holds with high probability, we want to apply a union bound. Therefore, let us fix disjoint sets $U$ and $Y$ of size $\eps n/\log n$, as well as disjoint set $N_u$, for each $u\in U_M$, of size $\log n$ so that all chosen sets are disjoint. Also, choose two colors $i_u$ and $j_u$ for each $u\in U_M$. There are at most $Z\leq \left(n^3 \binom{n}{\log n}\right)^{\eps n/\log n}n^{\eps n/\log n}$ ways to make all of these choices. We can simplify this bound as follows
\begin{align*}
    Z\leq n^{4\eps n/\log n} \left(\frac{en}{\log n}\right)^{\log n\cdot \frac{\eps n}{\log n}} \leq e^{5\eps n} \left(\frac{n}{\log n}\right)^{\eps n}.
\end{align*}

The probability that (A) is violated can therefore be bounded by summing the probabilities that $U$ and the sets $N_u$, chosen as above, indeed form a star-system and that no set $N_u$ has an edge to $Y$ in color $j_u$. The probability that $U$ and the sets $N_u$ form a star-system is $p^{\log n \cdot |U|}=p^{\eps n}$, since we need an edge in $G_{i_u}$ between $u$ and each vertex of $N_u$, for all $u\in U$. The probability there are no edges from $N_u$ to $Y$ in color $j_u$ for any $u$ is $(1-p)^{\log n\cdot |Y|\cdot |U|}\leq \exp(-p \eps^2 n^2/\log n)$. Putting all of this together, we have
\begin{align*}
\Pr[\text{(A) is violated}] &\leq e^{5\eps n}\left(\frac{n}{\log n}\right)^{\eps n} p^{\eps n}e^{-p \eps^2 n^2/\log n}\\
&\leq e^{5\eps n}\left(\frac{n}{\log n}\right)^{\eps n}\left(\frac{C\log n}{n}\right)^{\eps n} e^{-\eps^2 Cn}\\
&=e^{5\eps n}C^{\eps n}e^{-\eps^2 Cn}\to 0,
\end{align*}
as long as $C$ is large enough compared to $\eps$.

\medskip
\noindent
\textbf{Showing that (B) and (C) hold with high probability.}
We show that, with high probability, for every color-set $W\subseteq V\times [n]$ of size $\eps n/\log n$ and every set $U\subseteq V$ of size at least $n/2^7$, $N(W)$ intersects $U$. 

Note that this property implies both (B) and (C). If (B) was violated for some $X$ and $W$ with $|W|\geq \eps n/\log n$, we can set $U=L\backslash (X\cup V(W)\cup N(W))$ and observe that $|U|=|L| - |X| - \big|\big(V(W)\cup N(W)\big)\cap (L \setminus X)\big|\geq n/2-\eps n-n/3\geq n/2^7$. Since $N(W)\cap U=\varnothing$, this would violate the above property, thus showing (B) holds. To derive (C), we can set $W=\{(v, i):v\in S\}$ and $U=T$, and note that $N(W)\cap T\neq \varnothing$ implies that there is an edge in $G_i$ between $S$ and $T$.

Let us show that this new property holds with high probability. Fix such $W\subseteq V\times [n]$ with $|W|\geq \eps n/\log n$, $U\subseteq V$ with $|U|\geq \eps n$, and let $U'=U\setminus V(W)$. Note that $|U'|\geq \eps n/2$, as long as $\log n\geq 2$. Then, the probability that $N(W)$ does not intersect $U'$ can be bounded by $(1-p)^{|W||U'|}\leq e^{-p|W||U'|}$, since for each $u\in U', (w, i)\in W$, the edge $uw$ must not appear in $G_i$, which happes with probability $1-p$. By a union bound over all choices of $W$ and $U$, we get
\begin{align*}
    \Pr\big[\text{there are } U, W \text{ with }|U|\geq \eps n, |W|\geq \frac{\eps n}{\log n}\text{ s.t. }N(W)\cap U=\varnothing\big] &\leq \binom{n^2}{\eps n/\log n}\binom{n}{\eps n}\cdot e^{-p|W||U'|}.
\end{align*}
Using the usual bounds on the binomial coefficients, $\binom{n^2}{\eps n/\log n}\leq n^{2\cdot \eps n/\log n}\leq e^{2\eps n}$ and $\binom{n}{\eps n}\leq \big(\frac{en}{\eps n}\big)^{\eps n}\leq (e/\eps)^{\eps n}$, we get
\[\binom{n^2}{\eps n/\log n}\binom{n}{\eps n}
\cdot e^{-p|W||U'|}\leq e^{2\eps n}\left(e/\eps\right)^{\eps n}\cdot e^{-\eps^2 Cn/2}\to 0,\]
as long as $C$ is large enough compared to $\eps$.

\medskip
\noindent
\textbf{Setting up the deterministic part of the proof.} From now on, we will assume that the properties (A), (B) and (C) hold for $(G_1, \dots, G_n)$. Our covering procedure will have two steps. The first step will be to economically cover most of $X$ using a path which alternates between $L\backslash X$ and $X$, leaving at most $3\eps n/\log n$ vertices of $X$ uncovered. At this point, we will switch to the Friedman-Pippenger embedding, which will allow us to cover the rest of the vertices in a less economical way, which we can tolerate because we have to embed only a very small number of remaining vertices.

Let us fix a map $\chi:[n]\to [n]$ and a set $X\subseteq V$ we would like to cover with $|X|\leq \eps n$. Recall that the goal is to find a $\chi$-colored path in $(G_1, \dots, G_n)$ covering $X$, with endpoints in $L\backslash X$.

Since our path is required to use vertices of $X\cup L$, let us restrict our focus to $V'=X\cup L$ and set $G_i' = G_i[V']$. Also, in order to guarantee the expansion properties needed for the Friedman-Pippenger embedding, we need to remove the set of non-expanding color-vertex pairs from $V'$. To this end, define $m=n/20\log n$ and $D=\log n+1$, where the choice of $m$ and $D$ is guided by the requirements of Lemma~\ref{lem: forward}. We will use Lemma~\ref{lem: forward} on forests with maximum degree $D$ and the size of these forests will never exceed $n/20\leq mD$. Let $W\subseteq V'\times [n]$ be a maximal color-set such that $|(V(W)\cup N(W))\cap (L \setminus X)|\leq \min\{3D|W|,n/3\}$. By \ref{itm: non-expanding set}, $|W|\leq\eps n/\log n$ and hence, $|(V(W)\cup N(W))\cap (L \setminus X)|\leq 3D|W|\leq 3\eps n$. 

The goal of defining such $W$ is that we can apply the Friedman-Pippenger machinery with the ground set $A= (V'\times [n])\setminus W$, which has good expansion properties. In the economical covering part, we will strive to cover most of $X\backslash V(W)$, and therefore we set $L' = L\setminus (X\cup V(W))$, $X' = X\setminus V(W)$. 

Let $F$ be an empty graph on $|X\cup V(W)|$ vertices and let $\varphi$ be a bijection between $V(F)$ and $X\cup V(W)$. Now, we show that $\varphi$ is an $(m, D, A)$-good embedding  as well as that $(G'_1,\ldots,G'_n)$ is a $(2m,2D+1,A)$-color expander. 

Let us first verify the expansion property for a color-set $S\subseteq A$ of size at most $2m$. By the maximality of $W$, we have $|(V(S\cup W)\cup N(S\cup W))\cap (L \setminus X)|> \min\{3D(|S|+|W|),n/3\}$. Using that $|(V(W)\cup N(W))\cap (L \setminus X)|\leq 3D|W|$, we get 
\begin{align*}
    |N(S)\cap L'|> \min\{3D(|S|+|W|),n/3\}-3D|W|-|S|\geq \min\{(2D+1) |S|, n/5\}.
\end{align*}
To see the last inequality, one needs to analyze the two cases, depending on what the minimum on the left-hand side is. Firstly, if the minimum is $3D(|S|+|W|)$, then we have $3D(|S|+|W|)-3D|W|-|S|\geq (2D+1) |S|$, and otherwise $n/3-3D|W|-|S|\geq n/3-3\eps n-n/20\log n\geq n/5$. Hence, if $|S|\leq 2m$ then $|N(S)\cap L'|\geq (2D+1)|S|$, which shows that $(G'_1,\ldots,G'_n)$ is a $(2m,2D+1,A)$-color expander. 

On the other hand, to verify that $\varphi$ is an $(m,D,A)$-good embedding, we need to check that $|(N(S)\cap A)\backslash \varphi(F)|\geq \sum_{(x, i)\in S} D$ for all $S\subseteq A$ with $|S|\leq m$, where we note that other terms from (\ref{eqn:goodness}) do not appear since $F$ has no edges. However, we already know that $|(N(S)\cap A)\backslash \varphi(F)|\geq |N(S)\cap L'|\geq (2D+1) |S|\geq D|S|$ for all $|S|\leq m$. This shows that $\varphi$ is $(m, D, A)$-good.

\medskip
\noindent
\textbf{Economically covering most of $X$.} If $|X'|\leq 2\eps n/\log n$, we already have very few vertices in $X$ and we do not need to perform the economical covering step, allowing us to just pass to the Friedmann-Pippenger embedding step. So, in this step we assume that $|X'|>2\eps n/\log n$.

We say that a path $P=v_1 \dots v_\ell$ is \textit{alternating} if $v_1\in L'\backslash X', v_2\in X', v_3\in L'\backslash X', \dots, v_\ell\in X'$ and $v_iv_{i+1}$ is of color $\chi(i)$ for all $i\leq \ell-1$. Note that $\ell$ must be even for alternating paths $P$. 

Let us now consider pairs $(P, M)$, where $P$ is an alternating path and $M$ is a star-system disjoint from $P$ of size at most $\eps n/\log n$, with the following properties:
     \begin{itemize}
        \item For each $u\in U_M$, $N_u\subseteq L'$ and $G_{j_u}'$ contains no edges between $N_u$ and $X'\setminus (V(P)\cup V(M))$;
        \item Consider the following forest $F_{P, M}$ which extends $F$. Add an out-edge of color $\chi(1)$ from $\varphi^{-1}(v_2)$ to a new vertex, and for each vertex $v$ of $F$ with $\varphi(v)=v_{2k}$ and $2k<\ell$, add an out-edge from $v$ to a new vertex in color $\chi(2k)$. Also, for each vertex $v$ of $F$ with $\varphi(v)=u\in U_M$, add $\log n$ out-edges from $v$ to new vertices in color $i_u$. 
        Then, we require that there exists an $(m,D,A)$-good embedding $\varphi_{P, M}$ of $F_{P, M}$ extending $\varphi$, where, for each even index $2k<\ell$, the new out-neighbor of $\varphi^{-1}(v_{2k})$ gets mapped to $v_{2k+1}$, and for each $u\in U_M$, $\varphi_{P, M}$ is a bijection between the new vertices of the star attached at $\varphi^{-1}(u)$ and $N_u$.
     \end{itemize}

     Note that at least some pairs $(P, M)$ satisfy these properties (e.g. if both $M$ and $P$ are empty).
     
    Let $(P,M)$ be a pair satisfying the above properties maximizing the size of $M$ and (with respect to the size of $M$ being maximal) the size of $P$. Note that in this maximizing pair, $P$ is not an empty path, since otherwise we could extend it as follows. Pick a vertex of $v_2\in X'\backslash V(M)$ such that $(v_2, \chi(1))\notin W$. This can be done since $X'\backslash V(M)$ is nonempty and $X'$ does not intersect $V(W)$. Then, add an edge from $\varphi^{-1}(v_2)$ to a new vertex of color $\chi(1)$ and use Lemma~\ref{lem: forward} to extend $\varphi$ by embedding the new edge, and denote its end by $v_1$. This shows that $P$ can always be extended to contain at least one edge.
    
    Furthermore, we will show that the size of $M$ is limited by the requirement $|M|\leq \eps n/\log n$, as the following claim shows.
    
     \begin{claim}
         $|M|=\eps n/\log n$.
     \end{claim}
     \begin{proof}
         Suppose towards a contradiction that $|M|<\eps n/\log n$, and let us denote the length of $P$ by $\ell-1$. Let $\varphi_{P, M}$ be the embedding of the corresponding forest $F_{P, M}$ and note that $|V(F_{P, M})|\leq |(X\cup V(W))| + |M|\cdot \log n + |V(P)|\leq 2\eps n+\eps  n+2\eps n\leq 5\eps n$. Let $F_{P, M}'\supseteq F_{P, M}$ be obtained by attaching $\log n$ out-edges of color $\chi(\ell)$ to $\varphi^{-1}(v_\ell)$. 
         
         By \cref{lem: forward}, there exists an $(m, D, A)$-good embedding $\varphi_{P, M}'$ of $F_{P, M}'$ extending $\varphi'_{P, M}$. Let $N$ denote the set of vertices to which the new vertices of $F_{P, M}'$ get embedded to by $\varphi_{P, M}'$. Note that we have $N\subseteq L'$, since all vertices of $X\cup V(W)$ are already in the image of $\varphi_{P, M}$, and $\varphi_{P, M}'$ must be injective.
         
         If $G_{\chi(\ell+1)}'$ contains an edge $uv$ with $u\in N$ and $v\in X'\setminus (V(P)\cup V(M))$, we want to get a contradiction by extending the path $P$ and not altering $M$. We can define $v_{\ell+1}=u, v_{\ell+2}=v$, thus obtaining a longer path $Q$ and make no changes to $M$. We need to verify that $(Q, M)$ still satisfy the relevant properties. The first one is immediate, since $M$ is unchanged and the set $X'\backslash (V(P)\cup V(M))$ shrinks. To check the second property, note that $F_{Q, M}$ is a subgraph of $F_{P, M}'$ and therefore it can be obtained from it by removing a number of leaves attached to $\varphi_{P, M}^{-1} (v_\ell)$ (to be specific, one needs to remove the leaves corresponding to $\varphi_{P, M}'^{-1}(N\backslash\{v_{\ell+1}\})$). Thus, we can use \cref{lem: backward} to remove, one by one, all leaves attached to $\varphi_{P,  M}^{-1}(v_\ell)$ except $\varphi_{P, M}'^{-1}(v_{\ell+1})$. This gives an embedding $\varphi_{Q, M}$ verifying the second condition and gives a contradiction to the maximality of $(P, M)$.
         
        So, suppose that $G_{\chi(\ell+1)}'$ does not contain an edge between $N$ and $X'\setminus (V(P)\cup V(M))$. Let us shorten the path $P$ by removing the last two vertices, i.e. define a new path $Q=v_1\ldots v_{\ell-2}$, and extend $M$ by adding to $U_M$ the vertex $v_\ell$, together with colors $i_{v_\ell}=\chi(\ell)$, $j_{v_\ell}=\chi(\ell+1)$ and the neighborhood $N$. If the pair $(Q, M')$ satisfies the two properties, it would constitute a contradiction to maximality of $(P, M)$, where we are using the assumption $|M|< \eps n/\log n$, and therefore $|M'|\leq \eps n/\log n$.
        
        Note that the first property is still satisfied, precisely because $G_{\chi(\ell+1)}'$ does not contain an edge between $N$ and $X'\setminus (V(P)\cup V(M))=X'\setminus (V(Q)\cup V(M'))$. On the other hand, the second property can be verified using \cref{lem: backward} to remove $\varphi_{P, M}'^{-1}(v_{\ell-1})$ from the embedding $\varphi_{P, M}'$ and thus obtain the embedding $\varphi_{Q, M'}$. Hence, we get a contradiction to our original assumption that $|M|< \eps n/\log n$.\claimproofend

Hence, we have $|M|=\eps n/\log n$. Since, for every $u\in U_M$, there are no edges of color $j_u$ from $N_u$ to $X'\backslash (V(P)\cup V(M))$, we conclude that the number of vertices of $X$ not covered by $V(M)\cup V(P)\cup V(W)$ is at most $\eps n/\log n$, from property \ref{itm: DFS-expansion}. Thus, if we let $Y=X\setminus V(P)=\{y_1,\ldots,y_h\}$, then we have $|Y|\leq 3\eps n/\log n$. Also, note that $|V(P)|\leq 2|V(P)\cap X|\leq 2|X|\leq 2\eps n$. This means that we covered all but at most $3\eps n/\log n$ vertices of $X$ by a path of length at most $2|X|$, which completes the step of economically covering most of $X$.

\medskip
\noindent
\textbf{Covering the rest of $X$ using Friedman-Pippenger embedding.} In this step, we extend the path $P$ to a path $Q$ which covers the set $Y$ and has length at most $12\eps n$. The first step will be to decide in which order we will cover the vertices of $Y$. More precisely, to each vertex $y_t\in Y$, we will associate an index $k_t\in \{\ell+2\log n, \ell+15\eps n\}$ such that $y_t$ will be the $k_t$-th vertex of $Q$. This index $k_t$ needs to be chosen in such a way that the pairs $(y_t, \chi(k_t-1))$ and $(y_t, \chi(k_t))$ are not in the poorly-expanding color-set $W$, and also such that no two values $k_s, k_t$ satisfy $|k_s-k_t|<2\log n$. However, since $|W|\leq \eps n/\log n$, this will be possible, as shown by the following claim.

\begin{claim}
There exists a set of indices $k_1, \dots, k_h\in \{\ell+2\log n, \ell+15\eps n\}$ with the following properties: for any $y_t\in Y$, the following vertex-color pairs do not belong to $W$, i.e. $(y_t, \chi(k_t-1)), (y_t, \chi(k_t))\notin W$, and $|k_s-k_t|\geq 2\log n$ for all $1\leq s <t\leq h$.
\end{claim}
\begin{proof}
We can find such indices $k_1, \dots, k_h$ greedily, one by one. Suppose we have already chosen the values for $k_1, \dots, k_{t-1}$. Let us mark all values $x\in \{\ell+2\log n, \ell+15\eps n\}$ which have $|x-k_s|< 2\log n$ for some $1\leq s<t$. In this way, at most $t\cdot 4\log n\leq 12\eps n$ values in this interval are marked. Furthermore, let us mark all values $x\in \{\ell+2\log n, \ell+15\eps n\}$ for which $(y_t, \chi(x-1))$ or $(y_t, \chi(x))$ belongs to $W$. This marks at most another $2\eps n/\log n\leq \eps n$ values. However, since the set of all possible values of $k_t$ has cardinality at least $15\eps n-2\log n\geq 14\eps n$, and we have marked at most $13\eps n$ values, there is an unmarked $k_t$ which we can choose. Doing this for $t=1, \dots, h$ completes the proof.
\claimproofend

Let us reindex the elements of $y$ so that the indices $k_t$ are increasing, and also define $y_0$ to be the last vertex of the path $P$, along with $k_0=\ell$. Now, the idea will be to join the vertices $y_{t-1}$ and $y_{t}$ for all $t$ by appropriately colored paths of length $k_{t}-k_{t-1}$, thus obtaining a single $\chi$-colored path. To do this, we define a sequence of forests $F_0\subseteq F_1\subseteq \dots\subseteq F_h$ as follows. The initial forest $F_0$ is a restriction of $F_{P, M}$ on the vertex set $\varphi_{P, M}^{-1}(V(P)\cup Y)$, and it comes with the $(m,D,A)$-good embedding $\varphi_0$ (which is defined as a restriction of $\varphi_{P, M}$ to $F_0$, and it is $(m,D,A)$-good by \cref{lem: backward}). Then, $F_t$ is inductively defined by attaching a directed path $P_{t-1}$ starting at the vertex $\varphi_0^{-1}(y_{t-1})$ of length $s=\log n-6$, such that the $i$-th edge of this path has color $\chi(k_{t-1}+i-1)$, and by also attaching a directed path $P_{t}'$ starting at the vertex $\varphi_0^{-1}(y_{t})$ of length $k_t-k_{t-1}-s-1$, such that the $i$-th edge of this path has color $\chi(k_{t}-i)$. This is illustrated in Figure~\ref{fig:forest F_t}. 

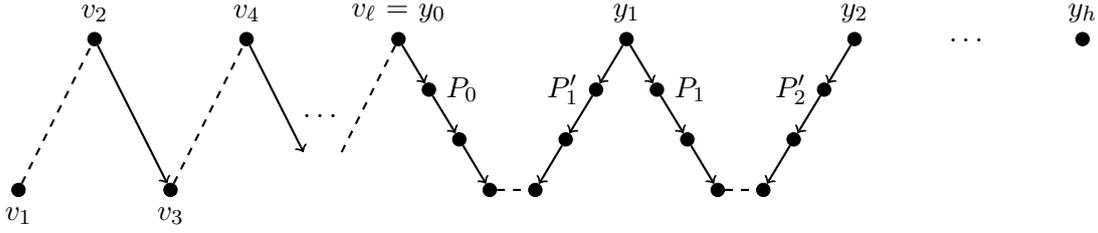
\begin{figure}
    \centering
    \begin{tikzpicture}
        \node[vertex, label=below:$v_1$] (v1) at (-4, -1) {};
        \node[vertex, label=above:$v_2$] (v2) at (-3, 1) {};
        \node[vertex, label=below:$v_3$] (v3) at (-2, -1) {};
        \node[vertex, label=above:$v_4$] (v4) at (-1, 1) {};
        \node[vertex, label=above:$v_\ell \text{ $=$ } y_0$] (vell) at (1, 1) {};
        \node[label=$\dots$] () at (0, -0.3) {};
        
        \draw[thick, dashed] (v1) -- (v2);
        \draw[thick, ->] (v2) -- (v3);
        \draw[thick, dashed] (v3) -- (v4);
        \draw[thick, ->] (v4) -- (-0.25, -0.5);
        \draw[thick, dashed] (vell) -- (0.25, -0.5);
        
        \node[vertex, label=above:$y_1$] (y1) at (4, 1) {};
        \node[vertex, label=above:$y_2$] (y2) at (7, 1) {};
        \node[label=$\dots$] () at (8.5, 0.7) {};
        \node[vertex, label=above:$y_h$] (yh) at (10, 1) {};

        \node[vertex, label=right:$P_0$] (p01) at (1.4, 0.33) {};
        \node[vertex] (p02) at (1.8, -0.33) {};
        \node[vertex] (p03) at (2.2, -1) {};

        \node[vertex, label=left:$P_1'$] (p1'1) at (3.6, 0.33) {};
        \node[vertex] (p1'2) at (3.2, -0.33) {};
        \node[vertex] (p1'3) at (2.8, -1) {};

        \draw[thick, ->] (vell) -- (p01);
        \draw[thick, ->] (p01) -- (p02);
        \draw[thick, ->] (p02) -- (p03); 
        \draw[thick, ->] (y1) -- (p1'1);
        \draw[thick, ->] (p1'1) -- (p1'2);
        \draw[thick, ->] (p1'2) -- (p1'3);
        \draw[thick, dashed] (p03) -- (p1'3);

        \node[vertex, label=right:$P_1$] (p11) at (4.4, 0.33) {};
        \node[vertex] (p12) at (4.8, -0.33) {};
        \node[vertex] (p13) at (5.2, -1) {};

        \node[vertex, label=left:$P_2'$] (p2'1) at (6.6, 0.33) {};
        \node[vertex] (p2'2) at (6.2, -0.33) {};
        \node[vertex] (p2'3) at (5.8, -1) {};

        \draw[thick, ->] (y1) -- (p11);
        \draw[thick, ->] (p11) -- (p12);
        \draw[thick, ->] (p12) -- (p13); 
        \draw[thick, ->] (y2) -- (p2'1);
        \draw[thick, ->] (p2'1) -- (p2'2);
        \draw[thick, ->] (p2'2) -- (p2'3);
        \draw[thick, dashed] (p13) -- (p2'3);
    \end{tikzpicture}
    \caption{Illustration of the construction of the forest $F_t$.}
    \label{fig:forest F_t}
\end{figure}

\begin{claim}\label{claim:embedding F_t}
For each $1\leq t\leq h$, there is an $(A,m,D)$-good embedding $\varphi_t$ of $F_t$ extending $\varphi_{t-1}$, with the property that the endpoints of $P_{t-1}$ and $P'_t$ are two vertices connected by an edge in $G'_{\chi(k_{t-1}+\log n)}$
\end{claim}

Proving this claim is sufficient to finish the proof. Namely, the forest $F_h$ is a union of paths covering $X$, and it is ensured that in the embedding $\varphi_h$, these paths can be joined into a $\chi$-colored path covering the whole of $X$. Furthermore, $v_1\in L'\setminus X'$ and by a simple application of Lemma~\ref{lem: forward}, we get that we can ensure that also the other endpoint is in $L'\setminus X'$. So, let us justify the claim then.

\begin{proof}[Proof of Claim~\ref{claim:embedding F_t}.]
Let us use Lemma~\ref{lem: forward} to embed all but the last $s=\log n-6$ vertices of the path $P_t'$ into $L'$ extending the embedding $\varphi_{t-1}$, thus arriving at a good embedding $\varphi'$ (ensuring that the edges are appropriately colored, of course). Now, we consider a forest $F_t'$, in which binary trees of depth $s$ directed away from the root are attached to $y_{t-1}$ and the last embedded vertices of $P_t'$, where all edges of the same depth receive the same color, so that any path from $y_{t-1}$ to any leaf has the same coloring as $P_{t-1}$ (and similar for $y_t$). This is illustrated in Figure~\ref{fig:binary trees}.
\begin{figure}[hbpt]
    \centering
    \begin{tikzpicture}[scale=0.7]
        \node[vertex, label=left:$y_{t-1}$] (v) at (-4, 0) {};

        \node[vertex] (v0) at (-3, -1) {};
        \node[vertex] (v1) at (-3, 1) {};
        \node[vertex] (v00) at (-2, -1.5) {};
        \node[vertex] (v01) at (-2, -0.5) {};
        \node[vertex] (v10) at (-2, 0.5) {};
        \node[vertex] (v11) at (-2, 1.5) {};
        \node[vertex] (v000) at (-1, -1.75) {};
        \node[vertex] (v001) at (-1, -1.25) {};
        \node[vertex] (v010) at (-1, -0.75) {};
        \node[vertex] (v011) at (-1, -0.25) {};
        \node[vertex] (v100) at (-1, 0.25) {};
        \node[vertex] (v101) at (-1, 0.75) {};
        \node[vertex] (v110) at (-1, 1.25) {};
        \node[vertex] (v111) at (-1, 1.75) {};

        \draw[very thick, brown] (v00) -- (v000);
        \draw[very thick, brown] (v00) -- (v001);
        \draw[very thick, brown] (v01) -- (v010);
        \draw[very thick, brown] (v01) -- (v011);
        \draw[very thick, brown] (v10) -- (v100);
        \draw[very thick, brown] (v10) -- (v101);
        \draw[very thick, brown] (v11) -- (v110);
        \draw[very thick, brown] (v11) -- (v111);
        
        \draw[very thick, Aquamarine] (v0) -- (v00);
        \draw[very thick, Aquamarine] (v0) -- (v01);
        \draw[very thick, Aquamarine] (v1) -- (v10);
        \draw[very thick, Aquamarine] (v1) -- (v11);

        \draw[very thick, ForestGreen] (v) -- (v0);
        \draw[very thick, ForestGreen] (v) -- (v1);
        
        \node[vertex] (u) at (3, 0) {};
        \node[vertex] (u0) at (2, -1) {};
        \node[vertex] (u1) at (2, 1) {};
        \node[vertex] (u00) at (1, -1.5) {};
        \node[vertex] (u01) at (1, -0.5) {};
        \node[vertex] (u10) at (1, 0.5) {};
        \node[vertex] (u11) at (1, 1.5) {};
        \node[vertex] (u000) at (0, -1.75) {};
        \node[vertex] (u001) at (0, -1.25) {};
        \node[vertex] (u010) at (0, -0.75) {};
        \node[vertex] (u011) at (0, -0.25) {};
        \node[vertex] (u100) at (0, 0.25) {};
        \node[vertex] (u101) at (0, 0.75) {};
        \node[vertex] (u110) at (0, 1.25) {};
        \node[vertex] (u111) at (0, 1.75) {};

        \draw[very thick, Goldenrod] (u00) -- (u000);
        \draw[very thick, Goldenrod] (u00) -- (u001);
        \draw[very thick, Goldenrod] (u01) -- (u010);
        \draw[very thick, Goldenrod] (u01) -- (u011);
        \draw[very thick, Goldenrod] (u10) -- (u100);
        \draw[very thick, Goldenrod] (u10) -- (u101);
        \draw[very thick, Goldenrod] (u11) -- (u110);
        \draw[very thick, Goldenrod] (u11) -- (u111);
    
        \draw[very thick, ForestGreen] (u0) -- (u00);
        \draw[very thick, ForestGreen] (u0) -- (u01);
        \draw[very thick, ForestGreen] (u1) -- (u10);
        \draw[very thick, ForestGreen] (u1) -- (u11);

        \draw[very thick, DarkOrchid] (u) -- (u0);
        \draw[very thick, DarkOrchid] (u) -- (u1);

        \draw[very thick, dashed, Lavender] (v010) -- (u110);

        \node[vertex, label=right:$y_{t}$] (u''') at (6, 0) {};
        \node[vertex] (u'') at (5, 0) {};
        \node[vertex, label=right:$\!\cdots$] (u') at (4, 0) {};

        \draw[very thick, RedOrange] (u''') -- (u'');
        \draw[very thick, Goldenrod] (u') -- (u);

        \draw [
            thick,
            decoration={
                brace,
                mirror,
                raise=0.5cm
            },
            decorate
        ] (u) -- (u''') 
        node [pos=0.5,anchor=north,yshift=-0.55cm] {\footnotesize ${k_t-k_{t-1}-2s-1}$}; 
    \end{tikzpicture}
    \caption{Illustration of the proof of Claim~\ref{claim:embedding F_t}.}
    \label{fig:binary trees}
\end{figure}
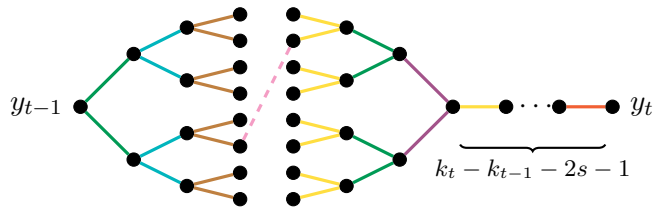

Again, by Lemma~\ref{lem: forward}, we can embed the forest $F_t'$ arriving at an $(m,D,A)$-good embedding $\varphi''$ which extends $\varphi'$. Observe that Lemma~\ref{lem: forward} applies since $F_t'$ has at most $2\cdot n/2^6+17\eps n\leq n/20\leq mD$ vertices (which is actually the inequality which constrains $\eps$ and requires, say, $\eps\leq 10^{-4}$).

Here, the key ingredient is the property \ref{itm:joined graph} - there is an edge of color $\chi(k_{t-1}+\log n)$ between two leaves from the two constructed trees. Having fixed these two leaves, we can delete all vertices other than their ancestors from the embedding using Lemma~\ref{lem: backward}, thus arriving at required embedding of the forest $F_t$. \claimproofend\qedhere
\end{proof}

\vspace{-0.7cm}
\section{Concluding remarks}

In this paper, we determined the threshold probability $p$ for an $n$-tuple of independent random graphs $(G_1, \dots, G_n)$ sampled from $G(n, p)$ to be Hamilton-universal. As we explained in the introduction, this result is a variant of a theorem by Bowtell, Morris, Pehova and Staden, who showed that any $n$-tuple of graphs $(G_1, \dots, G_n)$ with $\delta(G_i)\geq (1/2+o(1))n$ is Hamilton-universal. The fact that the $G_i$ are sampled independently in the random setting could lead to some intriguing differences.

For example, since the classical Dirac's condition for the containment of Hamilton cycle ($\delta(G)\geq n/2$) is tight, we must require at least the same condition from all graphs of the $n$-tuple in the deterministic setting. The reason for this is simple - if $G$ does not contain a Hamilton cycle, then $(G, \dots, G)$ does not contain a transversal Hamilton cycle either. Hence, in the transversal setting, the conditions for containment of a given structure cannot be weaker than the corresponding condition in the classical setting.

Once randomness is introduced, this behavior changes dramatically. For example, Anastos and Chakraborti \cite{AC23} showed that if $(G_1,\dots,G_n)$ is an $n$-tuple of independent graphs, where $G_i\sim G(n,p)$ with $p\gg \log n/n^2$, then $(G_1, \dots, G_n)$ contains a transversal Hamilton cycle with high probability. This differs starkly from the threshold probability for Hamiltonicity in $G(n,p)$ at $p=(1+o(1))\log n/n$. Therefore, the $n$-tuple $(G_1, \dots, G_n)$ contains a transversal Hamilton cycle even when no $G_i$ is expected to be Hamiltonian. This shows that looking for transversal structures in the random setting, as opposed to the deterministic setting, can be interesting even when none of the graphs $G_i$ have the required subgraph themselves. 

We wonder if we can observe a similar effect in questions regarding universality as well. Though Theorem~\ref{thm:main} is tight in the stated version, we can not rule out that it can be improved when only considering rainbow Hamilton cycles. More precisely, we say that a tuple $(G_1,\dots,G_n)$ of graphs on the same vertex set of size $n$ is \textit{rainbow Hamilton-universal} if it contains a $\chi$-colored Hamilton cycle for every bijection $\chi:[n]\to[n]$. Note the subtle difference to Hamilton-universality, where $\chi$ is not required to be bijective.
\begin{question}
What is the threshold probability $p$ such that the following holds with high probability? If $(G_1, \ldots, G_n)$ is an $n$-tuple of independent random graphs on the same vertex set sampled from $G(n,p)$, then $(G_1, \ldots, G_n)$ is rainbow Hamilton-universal.
\end{question}

The best lower bound we can prove is of the form $p \geq c/n$, for which we have two different arguments. Firstly, if $c<e$, then the expected number of copies of a Hamilton cycle with a specific rainbow color-pattern is $n!p^n=(1+o(1)) \sqrt{2\pi n}\big(\frac{n}{e}\big)^n\big(\frac{c}{n}\big)^n=o(1)$. Secondly, if $c<\log(2)$ then with high probability there exists a vertex $v$ which is isolated in more than half of the $G_i$. In this case, there exists a bijection $\chi:[n]\to[n]$ such that $v$ is isolated in $G_{\chi(i)}$ or $G_{\chi(i+1)}$ for all $i\in [n]$ and therefore, $(G_1,\dots,G_n)$ does not contain a $\chi$-colored Hamilton cycle. Curiously, the first lower bound has a better constant than the second, so that the expectation gives a better bound than the threshold for a specific vertex being bad. This is different from the usual behavior of Hamilton cycles in random graphs, where we can expect to find a Hamilton cycle as soon as every vertex has degree at least $2$. 

It would also be interesting to develop a universal analog to the results of \cite{AC23}. 
\begin{problem}
    Show that there exists $C$ such that for $p\geq C\log n/n$ the following holds with high probability. Let $(G_1, \ldots, G_n)$ be a tuple of graphs on the same vertex set of size $n$ such that $\delta(G_i)\geq (1/2+o(1))n$. For every $i$, let $F_i\subseteq G_i$ be obtained by keeping every edge independently with probability $p$. Then, $(F_1, \ldots, F_n)$ is Hamilton-universal.
\end{problem}

\noindent{\bf Acknowledgments.} We thank Katherine Staden for helpful comments on the manuscript. We also thank Ben Lund and Chuandong Xu for pointing out the similarity with \cite{lund2025embedding}.

\end{document}